\documentclass[11pt]{article}
\usepackage{amssymb, mathtools, amsthm}
\usepackage{fullpage}
\usepackage[
style=alphabetic,
maxbibnames=99,
giveninits=true]{biblatex}
\usepackage[ruled, vlined, linesnumbered]{algorithm2e}
\usepackage[normalem]{ulem} 
\usepackage{bbm}
\usepackage{breqn}
\usepackage[shortlabels]{enumitem}
\usepackage[all,cmtip]{xy}
\usepackage{xcolor}

\def\commenton{1}
\ifnum\commenton=1
\newcommand{\sw}[1]{\textcolor{blue}{{[Steve: #1]}}}
\newcommand{\syl}[1]{\textcolor{red}{{[Shuyao: #1]}}}
\else
\newcommand{\sw}[1]{}
\newcommand{\syl}[1]{}
\fi

\def\colorful{1}
\ifnum\colorful=1

\else

\fi

\newcommand{\la}{\leftarrow}

\newcommand{\ra}{\rightarrow}

\newcommand{\Ind}{\mathbf{1}}

\newcommand{\F}{\mathcal{F}} 
\newcommand{\G}{\mathcal{G}}

\newcommand{\E}{\mathbb{E}}

\newcommand{\Prob}{\mathbb{P}}

\newcommand{\R}{\mathbb{R}}
\newcommand{\Z}{\mathbb{Z}}

\newcommand{\X}{\mathcal{X}}
\newcommand{\Y}{\mathcal{Y}}
\newcommand{\grad}{\nabla}

\renewcommand{\>}{\rangle} 

\newcounter{thm}
\numberwithin{thm}{section}

\theoremstyle{plain}
\newtheorem{theorem}[thm]{Theorem}
\newtheorem{cor}[thm]{Corollary}
\newtheorem{lemma}[thm]{Lemma}
\newtheorem{prop}[thm]{Proposition}
\newtheorem{claim}[thm]{Claim}
\newtheorem{assump}[thm]{Assumption}
\newtheorem{fact}[thm]{Fact}

\theoremstyle{definition}
\newtheorem{defn}[thm]{Definition}

\theoremstyle{remark} 
\newtheorem{remark}[thm]{Remark}

\newcommand{\cXstar}{\mathcal{X}^\star}
\newcommand{\sigstarr}{{\sigma^\star_r}}
\newcommand{\sigstarl}{\sigma^\star_1}

\newcommand{\norm}[1]{\left\|{#1}\right\|}

\newcommand{\fnorm}[1]{\left\|{#1}\right\|_{F}}
\newcommand{\abs}[1]{\left|{#1}\right|}
\newcommand{\bigO}[1]{O\left({#1}\right)}
\newcommand{\bigtO}[1]{\tilde{O}\left({#1}\right)}
\newcommand{\Order}[1]{O\big({#1}\big)}
\newcommand{\tOrder}[1]{\tilde{O}\big({#1}\big)}

\newcommand{\epsg}{\epsilon_g}
\newcommand{\epsH}{\epsilon_H}

\DeclareMathOperator{\tr}{tr}
\DeclareMathOperator{\vect}{vec}
\DeclareMathOperator{\poly}{poly}
\newcommand{\inexact}{\hat}

\newcommand{\Ceps}{R_{\epsilon}}

\SetCommentSty{mycommfont}
\usepackage{hyperref}
\begin{document}

\title{A randomized algorithm for nonconvex minimization \\with inexact evaluations and complexity guarantees}


\author{Shuyao Li\footnote{Department of Computer Sciences, University of Wisconsin--Madison.}  \\ \url{shuyao.li@wisc.edu}\and Stephen J. Wright\footnotemark[1] \\\url{swright@cs.wisc.edu}}

\date{}

\maketitle
\begin{abstract}
  We consider minimization of a smooth nonconvex function with inexact oracle access to gradient and Hessian (without assuming access to the function value) to achieve approximate second-order optimality.
  A novel feature of our method is that if an approximate direction of negative curvature is chosen as the step, we choose its sense to be positive or negative with equal probability.
  We allow gradients to be inexact in a relative sense and relax the coupling between inexactness thresholds for the first- and second-order optimality conditions.
  Our convergence analysis includes both an expectation bound based on martingale analysis and a high-probability bound based on concentration inequalities.
  We apply our algorithm to empirical risk minimization problems and obtain improved gradient sample complexity over existing works.
\end{abstract}
\section{Introduction} 
\label{sec:intro}

Given a smooth function $f:\R^d \to \R$, we seek a point $x$ that approximately satisfies optimality conditions to be a local minimizer of $f$. 
Specifically, for small positive tolerances $\epsg$ and $\epsH$, we say that $x$ is an {\em $(\epsg,\epsH)$-approximate second-order point} when
\begin{equation} \label{eq:2on}
    \|\nabla f(x)\| \leq \epsilon_{g}, \quad \lambda_{\min }\left(\nabla^{2} f(x)\right) \geq-\epsilon_{H}.
\end{equation}
We define a simple algorithm for finding such a point and analyze its complexity, expressed in terms of $\epsg$ and $\epsH$, Lipschitz constants, failure probabilities, and algorithmic constants.
Our approach is based on a method using exact derivative information in which steps of only two types are taken: gradient descent steps and negative curvature steps for the Hessian. 
It allows gradient and Hessian evaluations to be {\em inexact}, where the gradient inexactness is bounded in a relative sense  (that is, as a fraction of the norm of the quantity being approximated).
Moreover, it does not require evaluation of the {\em value} of the objective function $f$\footnote{Approaches of this type are called ``objective-function-free optimization (OFFO)'' in \cite{gratton2023}}, which is an essential element of line-search and trust-region methods.
In such methods, the function value evaluation may need to be quite precise (and thus expensive and intolerant of noise~\cite{cartis2018global, chen2018stochastic, blanchet2019convergence, curtis2019stochastic, berahas2021global}) when the amount of decrease in $f$ over a step is small.
Instead, it makes use of conservative short steps to ensure descent.
Our strategy does not allow the algorithm to exploit local properties of the function (as opposed to its worst-case behavior). 

A distinctive feature of our method is that it ``flips a coin''  to decide whether to move in a positive or negative sense along a direction of negative curvature for an approximate Hessian.
(A more technical term for this strategy would be ``Rademacher randomness.'')
Because of this feature, our step may not yield descent in $f$ on every step; this fact must be accounted for in our analysis.
We can still obtain reasonable complexity results despite possible non-monotonicity because steps that increase $f$ are taken only when the gradient is small, so that the increase in $f$ at any such step is limited.

Our complexity results include both expected and high-probability stopping times for the algorithm; see Sections~\ref{sec:expected} and \ref{sec:strong-high-prob}, respectively.
Section~\ref{sec:roosta-work} gives a more detailed comparison with several prior works, chiefly \cite{yao2022inexact}, focusing on the requirements for gradient inexactness and on the usefulness of decoupling the tolerances $\epsg$ and $\epsH$.

In Section~\ref{sec:sampling}), we focus on the case in which $f$ has ``finite-sum" form
\begin{equation} \label{eq:fs}
    f(x) = \frac{1}{N} \sum_{i=1}^N f_i(x), 
\end{equation}
for a (typically large) value of $N$. 
We derive bounds on the number of samples (evaluations of $\nabla f_i(x)$ or $\nabla^2 f_i(x)$ for some $i$) required to meet the required level of inexactness for $\nabla f(x)$ and $\nabla^2 f(x)$ at each iterate.
Finally, we give a ``total sample-operation complexity" result, which is a bound on the total number of samples needed to find the point satisfying \eqref{eq:2on}, to high probability.

Finally, our approach to the analysis does not require a strong coupling between the tolerances $\epsg$ and $\epsH$.
We state several corollaries to our main result to show how various choices of these quantities, where both are expressed in terms of some underlying tolerance $\epsilon$, give rise to various overall complexity expressions. 
In prior works, a coupling of $\epsH = O(\sqrt{\epsg})$ is often assumed; we find situations where choices for which \(\epsH = o(\sqrt{\epsg})\) produce better overall bounds.



\subsection{Prior Work}
\label{sec:prior_work}
\paragraph{Approximate Second-Order Points.} Finding global minima of smooth nonconvex functions is intractable in general.
A common surrogate problem is to seek approximate second-order stationary points. 
This approach is advantageous in many practical settings where local minima are in fact global minima,  as in the case (under certain assumptions) in tensor decomposition~\cite{ge2015tensor}, dictionary learning~\cite{sun2016dictionary}, phase retrieval~\cite{sun2016geometric}, and low-rank matrix factorization problems~\cite{ge2017no-spurious-loc} such as sensing, completion, and robust PCA. 
See a discussion of this issue in~\cite[Section 1]{jin2017howto}.
A comprehensive technical review~\cite{zhang2020symmetry} discusses the loss landscape of problems in which there are no spurious local minima.

There are first-order methods that make perturbations to gradient descent~\cite{jin2017howto,jin2021nonconvex} or stochastic gradient descent~\cite{ge2015tensor, jin2021nonconvex}) directions to guarantee convergence to second-order stationary points. In some cases these methods require careful hyperparameter tuning and their iteration complexity can be high.

Second-order methods can make use of negative curvature directions of the Hessian (instead of random noise) to escape saddle points more efficiently. We focus on two types of computational complexity for this class of algorithms.
{\em Iteration complexity} describes the number of iterations needed to achieve the target accuracy, where each iteration may involve solving a difficult trust region problem or minimizing a cubic function.
{\em Operation complexity} describes the total number of gradient evaluations and  Hessian-vector products needed to attain target accuracy.

For trust-region methods,~\cite{cartis2012complexity} obtained iteration complexity $O(\max\{\epsH^{-1}\epsg^{-2}, \epsH^{-3}\})$ and~\cite{grapiglia2016worst} obtained $O(\max\{\epsg^{-3}, \epsH^{-3}\})$ iteration complexity. For operation complexity,~\cite{Cur19a} designed a Newton-CG trust region method that achieves \(\tOrder{\epsilon_{g}^{-7/4}}\), assuming \(\epsilon_{H} = \sqrt{\epsilon_{g}}\).

Cubic regularization~\cite{NesP06a} and its adaptive variants~\cite{cartis2011adaptive} find approximate second-order points with iteration complexity $O(\max\{\epsg^{-3/2}, \epsH^{-3}\})$ and $O(\max\{\epsg^{-2}, \epsH^{-3}\})$, respectively. 
With further assumptions on the cubic subproblems, \cite{cartis2012complexity} also obtained $O(\max\{\epsg^{-3/2}, \epsH^{-3}\})$ iteration complexity.
A recent algorithm described in~\cite{gratton2023} matches this complexity while also eliminating the need for function value evaluations, a property that our algorithm also enjoys.  
However, it lacks second-order guarantees when derivatives are inexact. 
Unique to this algorithm is its adaptability: Despite not requiring function evaluations, it does not require prior knowledge of Lipschitz parameters.
It achieves adaptivity by choosing steps based on a model function that becomes steadily more conservative as the algorithm progresses,  possibly even more conservative than the usual Taylor-series-based upper bound.
For the best operation complexity to our knowledge,~\cite{agarwal2017finding-approxi} achieved $\tOrder{\epsilon_g^{-7/4}}$, assuming $\epsH = \sqrt{\epsg}$.

Fixed-step and line-search methods are most relevant to our current paper. 
\cite{RoyW17a, royer2020a-newton-cg-alg} use a line search method that takes a Newton-CG step when the gradients are large and computes a negative curvature step when the gradients are small.
They obtain iteration complexity of \(\Order{\max\{ \epsg^{-3/2}, \epsH^{-3}\}}\) and \(\tOrder{\max\{\epsg^{-3}\epsH^{3}, \epsH^{-3}\}}\), respectively, and their operation complexity is a factor of \(\epsH^{-1/2}\) greater than their iteration complexity.
\cite{YCarmon_JCDuchi_OHinder_ASidford_2018a} similarly combined accelerated gradient descent and minimum eigenvector to achieve operation complexity \(\Order{\epsilon_g^{-7/4}}\), assuming $\epsH = \sqrt{\epsg}$.


\paragraph{Inexact derivatives.} 
There is increasing interest in finding approximate second-order stationary points when exact derivatives are not available. 
Some papers directly model inexactness in the gradient and Hessian oracles.
Others pay special attention to the sampling setting, where the objective function has a finite-sum structure \(f = \frac1N \sum_{i=1}^{N} f_{i}\) or stochastic structure $f(\cdot) = \E_{\omega} f(\cdot,\omega) $ and the inexactness comes from (batch) subsampling of \(f_{i}\) or $f(\cdot, w)$. Since the algorithms and their guarantees are quite different in the inexact and stochastic settings, we review these alternatives separately.

\subparagraph{Stochastic settings (including finite-sum problems~\eqref{eq:fs}).}
Specific to stochastic settings, we use the term ``total sample-operation complexity'' to describe the total number of gradient or Hessian-vector product evaluations of a single sample $f_i$.
 Perturbed stochastic gradient descent \cite{ge2015tensor, jin2021nonconvex} has total sample-operation complexity \(\tOrder{\epsg^{-4}}\) when $\epsH = \sqrt{\epsg}$.

\cite{arjevani2020second-order-in} obtained approximate second-order stationary points with total sample-operation complexity
        \(\Order{\epsg^{-3} + \epsg^{-2}\epsH^{-2} + \epsH^{-5}}\).
With access to stochastic gradients and stochastic Hessian-vector products, they improve on Natasha 2~\cite{allen-zhu2018natasha2} and use stochastic gradient descent with a variance reduction scheme when gradients are large, and Oja's method to compute a negative curvature direction when gradients are small. 
They also utilize a step of coin-flipping similar to ours to decide the sense of the negative curvature direction in their complex methods, but their inexactness is assumed to come only from sampling and they obtain only a constant-probability bound.
If access to full stochastic Hessians is assumed, they consider cubic regularization methods to remove the dependence on the gradient Lipschitz constant in their complexity bounds. 
(We do not focus on reducing the dependence on gradient or Hessian Lipschitz constants.)  

 \cite{bergou2021a-subsampling-l} considered batch subsampling
  settings, combining a Newton-type line search method with a routine
  that uses minimum eigenvectors to escape saddle points.
  They obtained  iteration complexity  $ \Order{\epsilon_{g}^{-3/2}}$ (assuming
  $\epsH = \sqrt{\epsg}$), while each iteration might require \(\Order{\epsH^{-1/2}}\)
  Hessian-vector products, resulting in \(\Order{\epsH^{-7/4}}\) operation complexity. Only an expectation bound on the iteration complexity was proved. 
  Application of this approach to batch subsampling setting is challenging because it requires more accurate access to function value oracle: for each iteration, they require $O(\epsg^{-2})$ samples to approximate the gradient, $O(\epsg^{-1})$ samples to approximate the Hessian --- and $O(\epsg^{-3})$ samples to approximate function values.

\subparagraph{Inexact settings beyond (sub)sampling.}

Some works have examined a setting in which the inexactness in gradient, Hessian, and sometimes function value oracles is not due to sampling. 
We use the term ``general inexact" to refer to this setting.
The stochastic settings discussed above can be viewed as a special case of the general inexact setting, because one can use $O(\delta_{g}^{-2})$ samples to construct an inexact gradient oracle and $O(\delta_H^{-2})$ samples to construct an inexact Hessian oracle with high probability, where $\delta_{g}$ and $\delta_{H}$ are the gradient and Hessian inexactness requirements.
(Details are given in  \cite{xu2020newton,tripuraneni2017stochastic,jin2023sample} and Section~\ref{sec:sampling}.)  
In addition, for all works reviewed below (but not our paper), $\delta_g = O(\epsg)$ and $\delta_H = O(\epsH)$. 
Since algorithms for general inexact settings do not consider the finite-sum or stochastic structure of the problem, the sample complexity and total sample-operation complexity derived in this way are generally not competitive with algorithms that are customized to stochastic settings. 
However, the concept of general inexactness is useful in the context of outliers and adversarial corruptions~\cite{li2023robust-second-order}, making it an important setup to consider.

 \cite{xu2020newton} analyzed trust-region methods and cubic regularization algorithms; they considered the settings where Hessians are inexact but gradients and function values are exact. 
  
\cite{tripuraneni2017stochastic} introduced a stochastic cubic regularization method that tolerates inexactness in both gradients and Hessians, assuming \(\epsH = \sqrt{\epsg}\). 
Although their results are stated only for stochastic settings, their analysis can be extended to the general inexact case.
They obtained iteration complexity $O(\epsg^{-3/2})$, which translates to $O(\epsg^{-7/4})$ operation complexity.  In batch subsampling settings, the total sample-operation complexity is $O(\epsg^{-7/2})$.
  
\cite{yao2022inexact} considers a fixed-step approach that has guarantees for finding approximate second-order stationary points assuming $\epsH = \sqrt{\epsg}$, where the gradients and Hessians are inexact. 
They obtained the same iteration complexity $O(\epsg^{-3/2})$ as in~\cite{tripuraneni2017stochastic}, and the operation complexity and the total sample-operation complexity in batch subsampling settings are therefore the same.
The settings are quite similar to ours, but we can tolerate more inexactness in the gradient, as discussed below in Section~\ref{sec:roosta-work}.
  
Concurrent with our work,~\cite{liu2023stochasticoptimization} have proposed a trust region method that tolerates both gradient and Hessian inexactness and achieves iteration complexity $O(\max\{ \epsH^{-1}\epsg^{-2}, \epsH^{-3}\})$.
However, this method requires either exact function value evaluation or else batch subsampling in a finite-sum setting to approximate the function value.
  
  \cite{paquette2020a-stochastic-li} considers the iteration complexity of finding an approximate first-order stationary point for convex and nonconvex objectives. This line search algorithm requires access to the approximate function value. \cite{jin2023sample} analyzed the oracle complexity and also considered trust region methods, and similarly obtained only first-order guarantees. We mention those papers because they allow for {\em relative} error in the gradient, as does our algorithm.

\subsection{Notation}

We take order notation $\Order{\cdot}$ and $o(\cdot)$ to have their usual meanings, and say that $f = \Theta(g)$ for two functions $f$ and $g$ if both $f = \Order{g}$ and $g = \Order{f}$. We say that \(f = \tOrder{g}\) when we wish to omit polylog factors involving inexactness tolerances (the $\epsg$ and $\epsH$ from \eqref{eq:2on}) or log factors involving failure probabilities in high-probability results (typically denoted by $\delta$ and $\xi$ below).


For an event \(A\), we use \(\Ind_A\) to denote the indicator random variable that equals 1 if the sample is in \(A\) and equals 0 otherwise. 
 We use  \(\norm{x}\) for the Euclidean norm of \(x\). 
 For \(d \in \Z_{+}\), we use \(I_{d}\) to denote the $d \times d$ identity matrix. 
 For matrix \(A\), we use \(\norm{A}\) and \(\fnorm{A}\) to denote the spectral norm and Frobenius norm of \(A\), respectively. 
When $A$ is symmetric,  we use \(\lambda_{\max}(A)\) and \(\lambda_{\min}(A)\) to denote its maximum and minimum eigenvalues, respectively, while \(\tr(A)\) denotes its trace. 
For square matrices $A$ and $B$ of the same dimension, We use \(\<A, B\> = \tr(A^{\top} B) \) to denote the entry-wise inner product.

\begin{defn}[Lipschitz continuity]
  A function \(h: \X \rightarrow \Y\) is \(\ell\)-Lipschitz if
  \[
  \norm{h(x_{1}) - h(x_{2})}_{\X} \le \ell \norm{x_{1} - x_{2}}_{\Y}\quad \mbox{for all $x_{1}, x_{2} \in \X$}.
  \]
  For purposes of this definition, we endow Euclidean spaces with the Euclidean norm. When $\X$ is a matrix space, we endow it with the Frobenius norm. When $\Y$ is a matrix space, we endow it with the spectral norm.
\end{defn}

\if false
\section{A  Simple Deterministic Algorithm}
\label{sec:simple-alg}
\include{simple_deterministic.tex}
\fi

\section{A randomized algorithm with expected descent}
\label{sec:rand-alg}

Here we define our algorithm (Algorithm~\ref{alg:inexact_randomized}) after stating our assumptions on $f$ and on the approximate gradients and Hessians.
We follow in Section~\ref{sec:expected} with a bound of the expected number of iterations required for the algorithm to terminate.
Section~\ref{sec:strong-high-prob} describes a high-probability bound, that is, a bound on the number of iterations of the algorithm required to stop at an approximate second-order point with probability at least $1-\delta$, for a specified (small, positive) value of $\delta$.

\subsection{Algorithm and Assumptions}

We begin with our assumptions on $f$ and on the approximate gradients and Hessians used by the algorithm.

\begin{assump}
  \label{assump:basics}
  The function $f$ has a gradient $\nabla f$ that is $L$-Lipschitz continuous at all points of interest and a Hessian that is $M$-Lipschitz continuous, and is uniformly bounded below by a quantity \(\bar{f}\).
\end{assump}
Under this assumption, the following polynomial bounds involving $f$ are well known.
\begin{fact}
  Assume \(f\) satisfies Assumption~\ref{assump:basics}. For any $x, d, p \in \R^{d}$ and $\alpha \in \R$, we have
  \begin{align}
    f(x+\alpha d) & \leq f(x)+\alpha \nabla f(x)^{T} d+\alpha^{2} \frac{L}{2}\|d\|^{2} \label{eq:taylor2nd} \\
    f(x+p) & \leq f(x)+\nabla f(x)^{T} p+\frac{1}{2} p^{T} \nabla^{2} f(x) p+\frac{1}{6} M\|p\|^{3}  \label{eq:taylor3rd}
\end{align}
\end{fact}

We suppose that when our algorithm requests a gradient evaluation at the iterate $x_k$, it receives the approximate value $g_k$, and  if it requests a Hessian, it receives a symmetric approximation $H_k$,
together with a unit-norm approximation to the minimum eigenvector $\inexact p_{k}$ and the value $\inexact \lambda_k := \inexact p_k^T H_k \inexact p_k$. 

We make the following assumption on the inexactness in $H_k$.
\begin{assump}
  \label{assump:inexactness}
  At all iterations \(k\), the errors in the gradient and Hessian oracles satisfy
  \[
\|g_{k} - \grad{f}(x_{k})\| \le \frac{1}{3} \max\{\epsilon_{g}, \|g_{k}\|\}, \quad \|H_{k} - \grad^{2} f(x_{k})\| \leq \frac29 \epsilon_{H}.
\]
\end{assump}
We could tolerate a greater degree of Hessian inexactness, specifically, \(\|H_{k} - \grad^{2} f(x_{k})\| \leq \frac29\max\{ \epsilon_{H}, |\lambda_{\min}(H_{k})|\}\).
However, since Hessian steps are taken infrequently in practice, this looser tolerance does not have great practical consequence, so we opt for the simpler analysis afforded by the bound of Assumption~\ref{assump:inexactness}.

Our algorithm requires an approximate direction of minimal curvature for the approximate Hessian $H_k$. 
Allowing for further error in the computation of this unit direction and the associated eigenvalue, we make an additional assumption.
\begin{assump}
Let \(\inexact p_{k}\) be an approximate unit eigenvector of \(H_{k}\) corresponding to the approximate minimum eigenvalue \(\inexact\lambda_{k}\), that is, \(\inexact\lambda_{k}:=\inexact p_{k}^{T}H_{k} \inexact p_{k}\). 
We assume that
  \label{assump:eigenvector_error}
  \[
  |\inexact\lambda_{k} - \lambda_{\min}(H_{k})| \le \frac19 \epsilon_{H}.
  \]
\end{assump}

This assumption enables practical computation of minimum eigenvalue and its corresponding eigenvector via only matrix-vector products. 
\cite[Appendix~B]{royer2020a-newton-cg-alg} used Lanczos and conjugate gradient to construct a minimum eigenvalue oracle satisfying Assumption~\ref{assump:eigenvector_error}.

We are ready to define our algorithm, which adds Rademacher randomness in the negative curvature steps to the second-order algorithm in~\cite[{Section~3.6}]{wright2021optimization-fo} and allows gradient and Hessian inexactness.


\begin{algorithm}[H]
  \label{alg:inexact_randomized}
  \caption{Algorithm using approximate gradients and Hessians to find \((\epsg, \epsH)\)-approximate second-order points}
  Given target tolerance $\epsg$, $\epsH$, Lipschitzness constants $L$, $M$, upper bound \(\alpha \in [\epsH, L]\) on negative curvature stepsizes, and initialization $x_0$ \\
      \For{$k=0,1,2,\dotsc$}{
      \uIf{\(\|g_{k}\| > \epsilon_{g}\)} {
        \tcc{approximate gradient  step}
        \(x_{k+1} = x_{k} - \frac1L g_{k}\)}
      \uElseIf{\(\inexact\lambda_{k} < -\epsilon_{H}\)}{
        \tcc{approximate negative curvature step}
        Choose \(\alpha_k \in [\epsH, \min(\alpha, | \hat\lambda_k|)]\)\\
        Draw \(\sigma_{k} \la \pm 1\) with probability \(\frac12\) \\\label{line:random}
        \(x_{k+1} = x_{k} + \frac{2 \alpha_{k}}{M} \sigma_k \inexact p_{k}\) \label{line:descent}
      }
      \Else{
        return \(x_{k}\)
      }
  }
\end{algorithm}

We observe the following output guarantee in
Algorithm~\ref{alg:inexact_randomized}:
\begin{prop}
If Algorithm~\ref{alg:inexact_randomized} terminates and returns \(x_{n}\), then \(x_{n}\) is an \((\frac43\epsilon_g, \frac{4}{3}\epsilon_{H})\) approximate second-order stationary point.
\end{prop}
\begin{proof}
    From the stopping criteria, we have
\(\norm{g_{n}} \le \epsilon_{g}\) and \(\inexact \lambda_{n} \ge -\epsilon_{H}\). Thus,
\[
  \norm{\grad f(x_{n})}
 \le \norm{g_{n}} + \frac13 \max\{\epsilon_{g}, \|g_{n}\|\} 
 \le \epsilon_{g} + \frac13 \epsilon_{g}  = \frac43 \epsilon_{g}.
\]
To bound \(\lambda_{\min}(\grad^{2} f(x_{n}))\), write \(U_{n} =  H_{n} - \grad^{2} f(x_{n})\), and note that $U_n$ is symmetric.
By Assumption~\ref{assump:inexactness}, \(\lambda_{\min}(-U_{n}) \ge -\norm{-U_{n}} \ge -\frac29 \epsilon_{H}\). By Assumption~\ref{assump:eigenvector_error}, \(\lambda_{\min}H_{n} \ge \inexact\lambda_{n} - \frac19\epsilon_{H} \ge -\epsilon_{H}  - \frac19\epsilon_{H} \ge -\frac{10}{9}\epsilon_{H}\).

We use Weyl's theorem to conclude that 
\[
  \lambda_{\min}(\grad^{2}f(x_{n}) ) \ge \lambda_{\min}(H_{n}) + \lambda_{\min}(-U_{n}) \ge -\frac{10}{9}\epsilon_{H}-\frac29 \epsilon_{H} = -\frac{4}{3}\epsilon_{H}, 
\]
as required.
\end{proof}

Note that this output guarantee only applies if Algorithm~\ref{alg:inexact_randomized} terminates. We proceed to derive the iteration complexity of this algorithm.
The proof technique is to find a lower bound on the amount of decrease per iteration. 
A complexity result then follows from the fact that \(f\) is bounded from below by \(\bar f\).

Write \(v_{k} :=  g_{k} - \grad f(x_{k})\). When the algorithm takes a gradient step, we have \(\|g_{k}\| > \epsilon_{g} \).
By Assumption~\ref{assump:inexactness}, we have \(\|v_{k}\| \le \frac13\|g_{k}\|\), so that  \(\|\grad f(x_{k}) \| \ge \frac23 \|g_{k}\|\). 
From~\eqref{eq:taylor2nd}, we have
\begin{align}
  f(x_{k+1}) & = f\left(x_{k} - \frac1L g_k\right)  \nonumber \\
             & \leq f(x_{k}) - \frac1L \grad f(x_{k})^{T}g_k + \frac{L}2 \cdot \frac1{L^{2}} \|g_k\|^{2} \nonumber \\
             & = f(x_{k}) - \frac1{2L}\left(\|\grad f(x_{k}) \|^{2} - \| v_{k}\|^{2} \right) \nonumber \\
             & \le f(x_{k}) - \frac1{2L} \left( \frac49\|g_{k}\|^{2} - \frac19\|g_{k}\|^{2} \right) \nonumber \\
             & \le f(x_{k}) - \frac1{6L}\|g_{k}\|^{2} \nonumber \\
             & \le f(x_{k}) - \frac1{6L}\epsilon_{g}^{2}, \label{eq:first-order-decrease}
\end{align}

For the negative curvature step, recall that \(\epsilon_{H} \le \alpha_{k} \le \alpha \le |\inexact \lambda_{k}| \le L\) and \(\|H_{k} - \grad^{2} f(x_{k})\| \leq \frac29\max\{ \epsilon_{H}, |\inexact\lambda_{k}|\}\). It follows from~\eqref{eq:taylor3rd} that
\begin{align*}
  f(x_{k+1})
  &= f(x_{k} + \frac{2\alpha_k}{M} \sigma_k \inexact{p}_{k})\\
  & \le f(x_{k}) + 2\frac{\alpha_k}{M}\grad f(x_{k})^{T}\sigma_k \inexact{p}_{k} + \frac12\frac{4\alpha_k^{2}}{M^{2}} \inexact{p}_{k}^{T} \grad^{2}f(x_{k})\inexact{p}_{k} + \frac{M}6  \frac{8\alpha_k^{3}}{M^{3}} \\
  & = f(x_{k}) + \frac{2\alpha_k^{2}}{M^{2}}\left(\inexact{p}_{k}^{T} \grad^{2}f(x_{k})\inexact{p}_{k} + \frac23 \alpha_k \right) + 2\frac{\alpha_k}{M}\grad f(x_{k})^{T}\sigma_k \inexact{p}_{k}
\end{align*}
When the algorithm takes a negative curvature step,  we have \(\inexact\lambda_{k} < -\epsilon_H < 0 \), so by  Assumption~\ref{assump:inexactness},  we have
\(\|H_{k} - \grad^{2} f(x_{k})\| \le \frac29 \epsH  \le \frac29 |\inexact\lambda_{k}|  \). It follows from the definition of spectral norm and Cauchy-Schwarz that
\(|\inexact{p}_{k}^{T}H_{k}\inexact{p}_{k} - \inexact{p}_{k}^{T}\grad^{2} f(x_{k})\inexact{p}_{k}| \le \|H_{k}\inexact{p}_{k} - \grad^{2} f(x_{k})\inexact{p}_{k}\|\le \frac{2|\inexact\lambda_{k}|}9\), so
\[
\inexact{p}_{k}^{T}\grad^{2} f(x_{k})\inexact{p}_{k} \le \frac29|\inexact\lambda_{k}| + {p}_{k}^{T}H_{k}\inexact{p}_{k} = \frac29 |\inexact\lambda_{k}| + (-|\inexact\lambda_{k}|) = -\frac79 |\inexact\lambda_{k}|.
\]
We thus have
\begin{align}
  f(x_{k+1})
  & \le f(x_{k}) + \frac{2\alpha_k^{2}}{M^{2}}\left(\inexact{p}_{k}^{T} \grad^{2}f(x_{k})\inexact{p}_{k} + \frac23 \alpha_k \right) + 2\frac{\alpha_k}{M}\grad f(x_{k})^{T}\sigma_k \inexact{p}_{k}  \nonumber \\
  & \le f(x_{k}) + \frac{2\alpha_k^{2}}{M^{2}}\left(-\frac79 |\inexact\lambda_{k}| + \frac23 \alpha_{k}\right) + 2\frac{\alpha_k}{M}\grad f(x_{k})^{T}\sigma_k \inexact{p}_{k} \nonumber \\
  & \le  f(x_{k})  - \frac{2\epsilon_{H}^{3}}{9M^{2}} + 2\frac{\alpha_k}{M}\grad f(x_{k})^{T}\sigma_k \inexact{p}_{k}, \label{eq:NC-second-order-decrease}
\end{align}
where the final inequality follows from our assumption $\alpha_k \in [\epsH, | \inexact\lambda_{k} |]$.

\subsection{Bound on expected stopping time}
\label{sec:expected}

We have the following result for expected stopping time of Algorithm~\ref{alg:inexact_randomized}.  
Here the expectation is taken with respect to the random variables $\sigma_k$ used at the negative curvature iterations.
For purposes of this and later results, we define 
\begin{equation} \label{eq:Ceps}
\Ceps:= \min\left(\frac{\epsilon_{g}^{2}}{6L},\frac{2\epsilon_{H}^{3}}{9M^{2}}\right).
\end{equation}

\begin{theorem}
  \label{thm:expected_complexity}
   Suppose the function \(f\) satisfies Assumption~\ref{assump:basics}, its inexact gradient and Hessian computations satisfy   Assumption~\ref{assump:inexactness}, and minimum eigenvalue computation satisfies Assumption~\ref{assump:eigenvector_error}. 
    Let \(T\) denote the iteration at which  Algorithm~\ref{alg:inexact_randomized} terminates. Then \(T < \infty\) almost surely and
    \begin{equation}
      \label{eq:expectation_bound}
      \E T \le \frac{f(x_{0}) - \bar f}{\Ceps},
    \end{equation}
    where $\Ceps$ is defined in \eqref{eq:Ceps}.
  \end{theorem}

  The proof of this result is given below.
It constructs a supermartingale \footnote{A supermartingale with respect to filtration \(\{\G_{1}, \G_{2}, \dotsc\}\) is a sequence of random variables $\{Y_1,Y_2, \dotsc \}$ such that for all \(k \in \Z_{+}\), (i) \(\E \abs{Y_{k}} < \infty\), (ii) \(Y_{k}\) is \(\G_{k}\)-measurable, and (iii) $\E (Y_{k+1} \, | \G_{k})  \le Y_k$.} based on the function value and uses a supermartingale convergence theorem and optional stopping theorem to obtain the final result.
  A similar proof technique is used in \cite{bergou2021a-subsampling-l} but for a line-search algorithm. 
  We collect several relevant facts about supermartingales before proving the result.

First, we need to ensure the relevant supermartingale is well-defined even after the algorithm terminates, so that it is possible to let the index \(k\) of the supermartingale go to \(\infty\).
  \begin{fact}[{\cite[Theorem 4.2.9]{durrett2019probability}}]
    \label{fact:stopped_martingale}
If \(T\) is a stopping time and \(X_{k}\) is a supermartingale, then \(X_{\min(T, k)}\) is a supermartingale.
\end{fact}
The following supermatingale convergence theorem will be used to ensure the function value converges, so that the algorithm terminates with probability 1.
\begin{fact}[Supermartingale convergence theorem, {\cite[Theorem 4.2.12]{durrett2019probability}}]
  \label{fact:martingale_convergence}
If \(X_{k} \ge 0\) is a supermartingale, then as \(k \ra \infty\), there exists a random variable \(X\) such that \(X_{k} \ra X \; a.s.\) and \(\E X \le \E X_{0}\).
\end{fact}

Finally, we will use optional stopping theorem to derive the expected iteration complexity. Note that we use a version of the optional stopping theorem specific to nonnegative supermartingales that does not require uniform integrability.
\begin{fact}[Optional stopping theorem, {\cite[Theorem 4.8.4]{durrett2019probability}}]
  \label{fact:optional_stopping}
  If \(X_{n}\) is a nonnegative supermartingale and \(N \le \infty\) is a stopping time, then \(\E X_{0} \ge \E X_{N}\)
\end{fact}

\begin{proof}[Proof of Theorem~\ref{thm:expected_complexity}]
We first construct a supermartingale \(M_{k}\) based on function values. 
We need to select a filtration to define this supermartingale formally.
We define the filtration \(\{\G_{k}\}\) to be the filtration generated by \(\sigma_{i},\; i \le k - 1\), the randomness used to compute the minimum eigenvalue and eigenvector of \(H_{i}, \; i \le k\),  the randomness (if any) used by the inexact gradient and Hessian oracles to compute \(g_{i}\) and \(H_{i}, \; i \le k\), and the randomness (if any) used to compute the stepsize \(\alpha_{i}, \; i \le k\).
It follows that \(g_{k}, x_{k}, H_{k}, \inexact \lambda_{k}, \inexact p_{k}, \alpha_{k} \in \G_{k}\) and \(\sigma_{k} \notin \G_{k}\).

Since \(\E \left[\sigma_k \right] = 0\), we have
\(\E \left[2\frac{\alpha_{k}}{M}\grad f(x_{k})^{T}\sigma_k \inexact{p}_{k}\middle\vert \G_{k}\right] = 0\).
We therefore have from \eqref{eq:NC-second-order-decrease} that
\begin{align*}
  \E\left[f(x_{k+1}) \middle\vert \G_{k}\right]
  & \le f(x_{k})- \frac{2\epsilon_{H}^{3}}{9M^{2}}.
\end{align*}
By combining with the minimum decrease \eqref{eq:first-order-decrease} obtained when the gradient step is taken, we have
\[\E\left[f(x_{k+1}) \middle\vert\G_{k}\right] \le f(x_{k})- \min\left(\frac{\epsilon_{g}^{2}}{6L},\frac{2\epsilon_{H}^{3}}{9M^{2}}\right) = f(x_k) - \Ceps.   \]


Consider the stochastic process \(M_{k}:= f(x_{k}) + k \Ceps\). Since \(x_{k} \in \G_{k}\), we also have \( M_{k} \in \G_{k} \). It follows that 
\begin{align*}
  \E\left[M_{k+1}\middle\vert \G_{k} \right]
 &= \E\left[  f(x_{k+1}) + (k+1) \Ceps \middle\vert \G_{k}\right] \\
  & \le \E\left[  f(x_{k})- \Ceps + (k+1) \Ceps \middle\vert \G_{k}\right]\\
  & =  \E\left[  f(x_{k}) + k \Ceps \middle\vert\G_{k}\right] = M_{k}.
\end{align*}

Hence, \(\{M_{k}\}\) is a supermartingale with respect to filtration \(\{\G_{k}\}\).


Let \(T\) denote the iteration at which our algorithm stops. 
Since the decision to stop at iteration \(k\) depends on \(g_{k}\) and \(\inexact\lambda_{k}\), we have \(\{T = k \} \in \G_{k}\), which implies \(T\) is a stopping time.

We will use the supermartingale convergence theorem (Fact~\ref{fact:martingale_convergence}) to show that \(T < +\infty\) almost surely, since the function value cannot decrease indefinitely as it is bounded from below by \(\bar f\). To apply Fact~\ref{fact:martingale_convergence}, we must be able to allow \(k \ra \infty\), so we need to transform \(\{M_k\}\) into a supermartingale \(\{Y_k\}\)  that is well defined even after the algorithm terminates.

It follows from Fact~\ref{fact:stopped_martingale} that \(Y_{k} :=  M_{\min(k,T)}\) is also a supermartingale. Since \(Y_{k} \ge \bar f\), it follows from the supermartingale convergence theorem (Fact~\ref{fact:martingale_convergence}) applied  to \(Y_{k} - \bar f\) that \(Y_{k} \rightarrow Y_{\infty}\) almost surely for some random variable \(Y_{\infty}\) with \(\E Y_{\infty} \le \E Y_{0} = \E M_{0} = f(x_{0}) < \infty\). Hence \(\Prob[Y_{\infty} = +\infty] = 0\). On the other hand, as \(k \rightarrow \infty\), we have \(k \Ceps  \rightarrow \infty\), so \( T = +\infty \implies Y_{k} = M_{k} \ge \bar f + k \Ceps \ra \infty \implies Y_{\infty} = +\infty\). Therefore we have \(\Prob[T < +\infty]= 1\).

We can then apply the optional stopping theorem (Fact~\ref{fact:optional_stopping}) to \(Y_{k} - \bar f\). It follows that \[ \bar f +  \E T \cdot \Ceps  \le \E f(x_{T}) +  \E T \cdot \Ceps  =\E\left[ M_{T} \right] = \E\left[ Y_{T} \right] \le \E [Y_{0}] = \E\left[ M_{0} \right] = f(x_{0}), \] where the first equality uses \(T < +\infty\) almost surely and the last inequality is Fact~\ref{fact:optional_stopping}.
By reorganizing this bound, we obtain \(\E T \le  \frac{f(x_{0}) - \bar f}{\Ceps} \), as desired.
\end{proof}

\begin{remark}
  The algorithm in~\cite[Section 3.6]{wright2021optimization-fo} also finds an \((\epsg, \epsH)\)-approximate second-order stationary points without allowing for gradient and Hessian inexactness. 
  Consequently, the algorithm there is deterministic; it does not need to introduce Rademacher randomness in the negative curvature step.
  The iteration complexity for that algorithm is the same as the bound on $\E T$ in Theorem~\ref{thm:expected_complexity} to within a constant factor.

  In summary, the expected iteration complexity of Algorithm~\ref{alg:inexact_randomized} matches that of the corresponding simple, deterministic, exact algorithm up to a constant factor, but it has the advantage of allowing for gradient and Hessian inexactness.
\end{remark}

\subsection{High-probability bound on stopping time}
\label{sec:strong-high-prob}

\label{sec:general-case}

In this section, we state and prove our main result concerning the number of iterations required for  Algorithm~\ref{alg:inexact_randomized} to terminate, with probability $1-\delta$, for some small positive $\delta$.
We start by stating the main result, then state several immediate corollaries concerning various choices of the parameters $\epsg$, $\epsH$, and $\alpha_k$. 


Recall that the gradient and negative curvature steps in Algorithm~\ref{alg:inexact_randomized} have the form
\begin{align}
\label{eq:gd-step}
x_{k+1} &= x_k - \frac{1}{L} g_k, \\
\label{eq:nc-step}
  x_{k+1} &= x_{k} + \frac{2 \alpha_{k}}{M} \sigma_{k} \inexact p_{k},
\end{align}
respectively, where \(\sigma_{k} = \pm 1\) with probability 1/2, and \(\alpha_k \in [\epsH, \min(\alpha, | \hat\lambda_k|)]\). 

\begin{theorem}
  \label{thm:high_prob_general}
  Suppose the function \(f\) satisfies Assumption~\ref{assump:basics}, its inexact gradient and Hessian computations satisfy Assumption~\ref{assump:inexactness}, and the minimum eigenvalue/eigenvector computation satisfies Assumption~\ref{assump:eigenvector_error}. 
  Then Algorithm~\ref{alg:inexact_randomized} terminates within at most \(n = 2\left( \frac{f(x_{0}) - \bar f}{\Ceps}\right) + B K  \) iterations with probability at least \(1-\delta\), where $\Ceps$ is defined in \eqref{eq:Ceps} and
  \begin{align*}
    K &= C \max \left\{ \log\left(\frac1\delta\right), \log \left( 4 \frac{f(x_{0}) - \bar f}{\Ceps} \right), 2 \log (2BC), 1 \right\}, \\
    B & = 1 + \frac{18 \alpha L}{M \epsilon_{g}}, \\
    C & = 2304 \frac{M^{2}\alpha^{2} \epsilon_{g}^{2}}{\epsilon_{H}^{6}}.
  \end{align*}  
\end{theorem}

The proof below will reveal the motivation behind the choice of quantities \(K\), \(B\), and \(C\). 
To write the result more compactly and suppress logarithmic dependence on $B, C$ and failure probability $\delta$, we can say
\begin{equation} \label{eq:hpb}
n = 2\left( \frac{f(x_{0}) - \bar f}{\Ceps}\right) + \tilde O\left( BC \right) = 2\left( \frac{f(x_{0}) - \bar f}{\Ceps}\right) + \tilde O\left( \frac{  M \alpha^{2} \epsg \max\left\{ L \alpha,M \epsilon_{g} \right\}}{\epsH^6}\right)
\end{equation}
We consider several choices of \(\epsg\) and \(\epsH\) and find their corresponding iteration complexity from this result. 
In the following statements, the expressions \(\Order{\cdot}\) and \(\tOrder{\cdot}\) hide dependence on the Lipschitz constants for gradient ($L$) and Hessian ($M$), initial function value gap $f(x_0) - \bar f$, and failure probability $\delta$.
In each case, we find conditions under which the iteration complexity is approximately $n = \tilde O(\epsg^{-2})$, which is within a logarithmic factor of the complexity of a gradient descent algorithm for finding a point satisfying approximate {\em first-order} optimality conditions.

\begin{cor}[Short-Step Negative Curvature Updates]\label{cor:short-step}
  Under the conditions of Theorem~\ref{thm:high_prob_general}, and setting \(\alpha_{k} = \epsH\) for all \(k\), we can define \(\alpha = \epsH\) in Algorithm~\ref{alg:inexact_randomized}. Then
  \begin{equation}
    \label{eq:hpb-short-step}
    n =
      6(f(x_{0}) - \bar f) \left( \max\left\{\frac{2L}{\epsg^{2}},  \frac{3M^{2}}{\epsH^{3}} \right\} \right) + \tilde O\left( \frac{ \max\{L \epsH, M \epsg\} M \epsg}{\epsH^{4}}\right)
  \end{equation}
  where the last term is dominated by the first term when \(M \epsg < L \epsH\)  up to a logarithmic factor and constants involving \(L, M\) and \(f(x_{0}) - \bar f\).
\end{cor}
\begin{proof}
    Equation \eqref{eq:hpb-short-step} follows immediately from \eqref{eq:hpb}. For the final claim, we have from  $M \epsg < L \epsH$  that the final term becomes $\tilde O\left( \frac{L^2}{\epsH^2} \right) = \tilde O(\epsH^{-2})$, so the first term  in \eqref{eq:hpb-short-step} dominates, as claimed.
\end{proof}

\begin{cor}
  When \( \epsilon_{g}= \epsilon, \epsilon_{H} = \sqrt{\epsilon M}\), and
  \(\alpha_{k} = |\inexact\lambda_{k}|\), 
  we can take \(\alpha = 2 L\). Then \(\Ceps = O(\epsilon^2)\),
  \(B = O(\epsilon^{-1})\), and \(C=O(\epsilon^{-1})\), hence \(K = \tilde O(\epsilon^{-1}) \) and \(n=\tilde O(\epsilon^{-2})\).
\end{cor}




\begin{remark} \label{remark:random_inexactness}
Theorem~\ref{thm:high_prob_general} and its analysis require that Assumption~\ref{assump:inexactness} holds for all iterations. 
If the probability that Assumption~\ref{assump:inexactness} fails on any given iteration is at most \(\xi\), we can use the union bound to ensure the probability that Assumption~\ref{assump:inexactness} holds for all $n$ iterations to be at least \(1-n \xi\).
If the failure probability tolerance for both types of failure (either that  Algorithm~\ref{alg:inexact_randomized} does not terminate or that Assumption~\ref{assump:inexactness} fails at some iterations) is \(\Theta(\delta)\) in total, then we require \(\xi = O(\delta/n)\). When, for example, \(n = O\left( \frac1{\epsilon^{2}} \log\left( \frac1\delta\right )\right ) = O\left( \frac1{\epsilon^{2}\delta} \right)\), the condition  \(\xi = O(\delta^{2} \epsilon^{2})\) ensures a total failure probability of \(\Theta(\delta)\).
  We will elaborate on this point in Section~\ref{sec:sampling} when we discuss using mini-batch sampling to achieve  Assumption~\ref{assump:inexactness}.
\end{remark}

\begin{proof}[Proof of Theorem \ref{thm:high_prob_general}]
For the $n$ given in Theorem~\ref{thm:high_prob_general} and \eqref{eq:hpb}, define the event 
\begin{equation} \label{eq:E}
E := \{\mbox{Algorithm does not terminate before step } n \}.
\end{equation}
We wish to show \(\Prob(E) \le \delta\).

In the discussions below, we use notation \(\{k \in NC\}\)  and \(\{k \in GD\}\) to indicate the event that \(k\)-th step (from \(x_{k}\) to \(x_{k+1}\)) is a negative curvature step and a gradient descent step, respectively. Let $|NC|$ be the number of negative curvature steps the algorithm has taken before termination, and let $|GD|$ be the number of total gradient descent steps. Note that \(|NC|\) and \(|GD|\) are both random variables whose values depend on all randomness in the algorithm.

We combine the inequalities \eqref{eq:first-order-decrease} and \eqref{eq:NC-second-order-decrease} to obtain the descent from both gradient descent and negative curvature steps:
\begin{align*}
  f(x_{k+1}) & \le f(x_{k}) + \Ind_{\{k\in NC\}} \left( - \frac{2\epsilon_{H}^{3}}{9M^{2}} +  \frac{2\alpha_{k}}{M}\sigma_{k}\grad f(x_{k})^{T}\inexact{p}_{k} \right) + \Ind_{\{k \in GD\}} \left( -\frac{\epsilon_{g}^{2}}{6L}\right ).
\end{align*}
By telescoping the above inequality, we obtain
\begin{align}
  \label{eq:combine_taylor_general}
  \begin{split}
     \bar f - f(x_{0})
    & \le f(x_{n}) - f(x_{0}) \\ & \le \sum_{k=0}^{n-1} \Ind_{\{k\in NC\}} \left( - \frac{2\epsilon_{H}^{3}}{9M^{2}} + \frac{2\alpha_{k}}{M}\sigma_{k}\grad f(x_{k})^{T}\inexact{p}_{k} \right) + \Ind_{\{k \in GD\}} \left( -\frac{\epsilon_{g}^{2}}{6L}\right ) \\
    & \le  |NC| \frac{2  \epsilon_{g}}{{M}}  \left(  - \frac{\epsilon_{H}^{3}}{9 M  \epsilon_{g}} + \frac{1}{|NC|}\sum_{k \in NC} \frac{\grad f(x_{k})^{T}\inexact{p}_{k}  \alpha_{k} }{\epsilon_{g}}  \sigma_{k} \right) + |GD|\left( -\frac{\epsilon_{g}^{2}}{6L} \right)
\end{split}
\end{align}

We first prove two claims that help us understand the event \(E\) defined in \eqref{eq:E} better.
The first claim implies that \(E \subset \{ |NC| \ge K\}\).
\begin{claim}
  \label{claim:K_general}
   For all \(K \in \Z_{+}\), if  the algorithm does not terminate within \(n=2\left( \frac{f(x_{0}) - \bar f}{\Ceps}\right) + B K\) steps (that is, event \(E\) is realized), then we have \(|NC| \ge K\) with probability 1.
\end{claim}
\begin{proof}[Proof of Claim \ref{claim:K_general}]
   Assume to the contrary that \(|NC| < K\). It follows from \(|GD| + |NC| = n = 2\left( \frac{f(x_{0}) - \bar f}{\Ceps}\right) + B K\) that \(|GD| > 2 \frac{f(x_{0}) - \bar f}{\Ceps} + (B-1) K >   \frac{f(x_{0}) - \bar f}{\Ceps} + \frac{18 \alpha L}{M \epsilon_{g}} K \).
Then
\begin{align*}
  f(x_{0}) - f(x_{n})
  & \ge \sum_{k=0}^{n-1} \Ind_{\{k\in NC\}} \left(  \frac{2\epsilon_{H}^{3}}{9M^{2}} - \frac{2\alpha_{k}}{M}\sigma_{k}\grad f(x_{k})^{T}\inexact{p}_{k} \right) + \Ind_{\{k \in GD\}} \left( \frac{\epsilon_{g}^{2}}{6L}\right ) \\
  & \ge |NC| \left( \frac{2\epsilon_{H}^{3}}{9M^{2}} - \frac{2\alpha}{M} \|\grad f(x_{k})\|  \right) + |GD| \frac{\epsilon_{g}^{2}}{6L} \\
  & \ge |NC| \left(  - \frac{2\alpha}{M} \frac43 \epsilon_{g} \right) + |GD| \frac{\epsilon_{g}^{2}}{6L} \\
  & > K \left( - \frac{8\alpha\epsilon_{g}}{3M} \right)  + \left( \frac{f(x_{0}) - \bar f}{\Ceps} + \frac{18 \alpha L}{M \epsilon_{g}} K  \right ) \frac{\epsilon_{g}^{2}}{6L}  \\
  & =  K \left( \frac{\epsilon_{g}\alpha}{3M} \right ) + \frac{f(x_{0}) - \bar f}{\Ceps}  \frac{\epsg^{2}}{6L} \\
  & > \frac{f(x_{0}) - \bar f}{\Ceps}  \frac{\epsilon_{g}^{2}}{6L} \ge f(x_{0}) - \bar f,
\end{align*}
where the second line is Cauchy-Schwarz, the third line uses the facts that \(\| \nabla f(x_k) \| \le \frac{4}{3} \epsg\) for iterations $k \in NC$ (see Assumption~\ref{assump:inexactness} and Algorithm~\ref{alg:inexact_randomized}), the fourth line uses the assumption on \(|NC|\) and \(|GD|\), and the final inequality is because \(\Ceps \le \frac{\epsilon_g^{2}}{6L}\) by definition. This inequality contradicts the lower bound of \(f\), so we must have $|NC| \ge K$, as claimed.
\end{proof}

The next claim concerns the size of the ``$NC$" term in \eqref{eq:combine_taylor_general}.
\begin{claim}
  \label{claim:concentrate_general}
  Assume that the algorithm does not terminate within \(n = 2\left( \frac{f(x_{0}) - \bar f}{\Ceps}\right) + B K \) steps. 
  Then in the coefficient of the ``$NC$" term in  \eqref{eq:combine_taylor_general}, we have
  \[  - \frac{\epsilon_{H}^{3}}{9 M  \epsilon_{g}} + \frac{1}{|NC|}\sum_{k \in NC} \frac{\grad f(x_{k})^{T}\inexact{p}_{k}  \alpha_{k}}{\epsilon_{g}}  \sigma_{k} >   - \frac{\epsilon_{H}^{3}}{18 M   \epsilon_{g}}\]
\end{claim}
\begin{proof}[Proof of Claim \ref{claim:concentrate_general}]
  Assume to the contrary that \[ - \frac{\epsilon_{H}^{3}}{9 M \epsilon_{g}} + \frac{1}{|NC|}\sum_{k \in NC} \frac{\grad f(x_{k})^{T}\inexact{p}_{k}\alpha_{k} }{\epsilon_{g}}  \sigma_{k} \le   - \frac{\epsilon_{H}^{3}}{18 M   \epsilon_{g}}.\]
By substituting this bound into \eqref{eq:combine_taylor_general} we obtain
  \begin{align*}
    \bar f - f(x_{0})
    & \le f(x_{n}) - f(x_{0})  \\
    & \le |NC| \frac{2  \epsilon_{g}}{{M}}  \left(   - \frac{\epsilon_{H}^{3}}{18 M  \epsilon_{g}} \right) + |GD|\left( -\frac{\epsilon_{g}^{2}}{6L} \right) \\
      & =   |NC|  \left(   - \frac{\epsilon_{H}^{3}}{9 M^{2}} \right) + |GD|\left( -\frac{\epsilon_{g}^{2}}{6L} \right) \\
    & \le  |NC|  \left( - \frac{\epsilon_{H}^{3}}{9M^{2}}\right)  + |GD| \left( -\frac{\epsilon_{g}^{2}}{12L}\right ) \\
    & \le - n \left(\frac12 \Ceps \right).
  \end{align*}
  where the final inequality follows from the definition of $\Ceps$ in \eqref{eq:Ceps} and the fact that $|GD|+|NC|=n$.
  We thus have  \(n \le  2\left( \frac{f(x_{0}) - \bar f}{\Ceps}\right)\), which contradicts \(n = 2\left( \frac{f(x_{0}) - \bar f}{\Ceps}\right) + B K  \) since $B>1$ and $K \ge 1$. 
\end{proof}

For the event $E$ defined in \eqref{eq:E},  Claim~\ref{claim:concentrate_general} implies that 
\[
  E  \subset \left\{  \frac{1}{|NC|}\sum_{k \in NC} \frac{\grad f(x_{k})^{T}\inexact{p}_{k} \alpha_{k} }{\epsilon_{g}}  \sigma_{k} >  \frac{\epsilon_{H}^{3}}{18 M \epsilon_{g}} \right\}
\]
By combining Claims~\ref{claim:K_general} and \ref{claim:concentrate_general},  we obtain \(E \subset F\) where 
\[ 
F:= \left\{  \frac{1}{|NC|}\sum_{k \in NC} \frac{\grad f(x_{k})^{T}\inexact{p}_{k} \alpha_{k}}{\epsilon_{g}}  \sigma_{k} >  \frac{\epsilon_{H}^{3}}{18 M \epsilon_{g}}  \right\} \cap \{ |NC| \ge K\}.
\] 
To show \(\Prob(E) \le \delta\), it suffices to show  \(\Prob\left( F \right) \le \delta\).

Let the subsequence \(k_{i}\) enumerate \(\{k: k\in NC\}\), that is,  \(\{k: k\in NC\} = \{k_{1}, \dots, k_{|NC|}\}\). (Formally, we can consider \(k_{i}\)'s as a sequence of random variables that take non-negative integer values).
We define 
\begin{equation} \label{eq:tau}
\tau_{i} :=  \frac{\grad f(x_{k_{i}})^{T}  \inexact p_{k_{i}} \alpha_{k_{i}}}{\epsilon_g} \sigma_{k_{i}},
\end{equation}
and extend the definition of \(\tau_i\) for $i>|NC|$ by defining \(\tau_{i} = 0\) for all such $i$.

Observe that realizations of the random variable \(|NC|\) on event \(\{|NC| \ge K \} \) can be \(\{K, K+1, \cdots, n\}\). Therefore, we have
\begin{align*}
F & \subset \left\{\exists s = K, K+1, \dots, n, \; \frac1s \sum_{i=1}^{s} \frac{\grad f(x_{k_{i}})^{T}\inexact{p}_{k_{i}} \alpha_{k_{i}}}{\epsilon_{g}}  \sigma_{k_{i}} >  \frac{\epsilon_{H}^{3}}{18 M \epsilon_{g}} \right\} \cap \{ |NC| \ge K\}\\
    & = \cup_{s=K}^{n} \left\{ \frac1s \sum_{i=1}^{s}\tau_{{i}} >  \frac{\epsilon_{H}^{3}}{18 M  \epsilon_{g}}\right\} \cap \{ |NC| \ge K\} \\
   & \subset  \cup_{s=K}^{n} \left\{ \frac1s \sum_{i=1}^{s}\tau_{{i}} >  \frac{\epsilon_{H}^{3}}{18 M \epsilon_{g}} \right\}.
\end{align*}



By the union bound, to show that \(\Prob(F) \le \delta\), it suffices to choose $K$ large enough that
\[\Prob\left\{ \frac1s \sum_{i=1}^{s}\tau_{{i}} >   \frac{\epsilon_{H}^{3}}{18 M  \epsilon_{g}} \right\} \le \frac{\delta}{n} \quad \mbox{for all $s =K, K+1, \dots, n$.}
\]


We proceed to apply Azuma's inequality to \(X_{t} := \sum_{i=1}^{t} \tau_{i} \)
and \(X_{0} = 0\). Let  \(\F_{t}\) be generated by Rademacher variables
\(\sigma_{k_{i}}\) for \(i\le t\), the randomness used to compute the minimum
eigenvalues and their corresponding eigenvectors of \(H_{j}\) for \(j \le k_{t+1}\), the randomness (if any) used by the inexact gradient and Hessian oracles to compute \(g_{j}\) and \(H_{j}, \; j \le k_{t+1}\), and the randomness (if any) used to compute the stepsize \(\alpha_{j}, \; j \le k_{t + 1}\). 
Recall the negative curvature update is
\(x_{k+1} = x_{k} + \frac{2 \alpha_{k}}{M} \sigma_k \inexact p_{k}\). For a
given \(i\) and for all \(k_{i} < j \le k_{i+1}\), we have that \(x_{j}\) are
\(\F_{i}\)-measurable, since the algorithm takes only gradient steps between
\(k_{i}\)-th and \(k_{i+1}\)-th iteration and no Rademacher randomness is involved in
gradient steps. 
In Algorithm~\ref{alg:inexact_randomized}, the stepsize \(\alpha_{k_{i}}\) is chosen before the randomness \(\sigma_{k_{i}}\) is drawn, hence we have \(\alpha_{k_{i}} \in \F_{i-1}\).

\begin{fact}[Azuma's inequality]
Let $\left\{X_{0}, X_{1}, \cdots\right\}$ be a martingale (or supermartingale) with respect to filtration $\left\{\mathcal{F}_{0}, \mathcal{F}_{1}, \cdots\right\}$. Assume there are constants $0 <c_{1}, c_{2}, \cdots<\infty$ such that
$$
 -\frac12 c_{t} \leq X_{t}-X_{t-1} \leq  \frac12 c_{t}
$$
almost surely. Then for all $\beta>0$, we have
$$
\Prob \left(X_{n}-X_{0} \geq \beta \right) \leq \exp \left(-\frac{2 \beta^{2}}{\sum_{t=1}^{n} c_{t}^{2}}\right).
$$
\end{fact}

Observe from \eqref{eq:tau} that
\[
\E[\tau_{i}\vert \F_{i-1}] = \E\left[  \frac{\grad f(x_{k_{i}})^{T}  \inexact p_{k_{i}}\alpha_{k_{i}} \sigma_{k_{i}}}{\epsilon_g} \middle\vert \F_{i-1} \right ]= \frac{\grad f(x_{k_{i}})^{T}  \inexact p_{k_{i}}\alpha_{k_{i}}}{\epsilon_g} \E \sigma_{k_{i}} = 0,
\]
where the second equality is because \(x_{k_{i}}\) is  \(\F_{i-1}\)-measurable. So \(\E[X_{t} \vert \F_{t-1}] = \E[X_{t-1} + \tau_t \vert \F_{t-1}] = X_{t-1} \), 
which implies that \(X_{t}\) is a martingale with respect to \(\{\F_{t}\}\).

From \eqref{eq:tau}, Assumption~\ref{assump:inexactness}, and the fact that \(\alpha_{k} \le \alpha\) for all \(k\), we have
\[
\left|\tau_{i} \right| \le \frac{\| \nabla f(x_{k_i}) \|}{\epsilon_g} \alpha \le \frac{\| g_{k_i} \| + \max(\epsilon_g, \| g_{k_i} \|)}{\epsilon_g}\alpha \le  \frac{4}{3}\alpha,
\]
where the last inequality follows from the fact that iteration $k_i$ is a  negative curvature step, so that $\| g_{k_i} \| \le \epsilon_g$. 
Thus Azuma's inequality implies that 
\begin{align*}
  \Prob\left\{\sum_{i=1}^{s}\tau_{i} >   \frac{s\epsilon_{H}^{3}}{18 M   \epsilon_{g}} \right\} \le \exp\left( - \frac{2 \left(  \frac{s\epsilon_{H}^{3}}{18 M \epsilon_{g}}\right )^{2} }{ \left(\frac{8}{3}\alpha\right)^{2}s  } \right ) = \exp\left( -\frac{s \, \epsilon_{H}^{6}}{1152 M^{2} \alpha^{2} \epsilon_{g}^{2}}\right ) \le \exp\left( -\frac{K \epsilon_{H}^{6}}{1152 M^{2} \alpha^{2} \epsilon_{g}^{2}}\right )
\end{align*}

Recall that \(C:= 2304 \frac{M^{2}\alpha^{2} \epsilon_{g}^{2}}{\epsilon_{H}^{6}}\). To make the above probability less than \(\frac{\delta}{n}\), it suffices that 
\begin{equation} \label{eq:Klb1}
K \ge \frac{C}{2} \log(n/\delta) =  \frac{C}{2} (\log(A + B K) + \log(1/\delta)),
\end{equation}
where \(A =  2\left( \frac{f(x_{0}) - \bar f}{\Ceps}\right)\).





We achieve \eqref{eq:Klb1} by choosing $K$ large enough that simultaneously \(\frac{K}{2} \ge \frac{C}{2} \log(\frac1\delta)\) and \(\frac{K}{2} \ge \frac{C}{2} (\log(A + B K)\).

For the latter lower bound $K \ge C \log(A+BK)$, it suffices that simultaneously $K \ge C \log (2A)$ and $K \ge C \log (2BK)$, where the latter is equivalent to
\begin{equation} \label{eq:Klb2}
K/C \ge \log(K/C) + \log (2BC).
\end{equation}
To derive an explicit sufficient condition for the latter, we use the following fact.

\begin{fact}
    Given any $V$ with $\log V \ge 1$, we have that $T - \log T \ge \log V$ provided that $T \ge 2 \log V$.
\end{fact}
\noindent \emph{Proof of Fact:} Note that $T - \log T$ is an increasing function of $T$ for $T \ge 1$, so it suffices to prove that  $T - \log T \ge \log V$ for $T = 2 \log V$. We know that $S \ge \log (2S)$ for all $S \ge 1$.  By setting $T = 2 \log V$ and $S = \log V$, we have
\[
T - \log T = 2 \log V - \log (2 \log V) \ge 2 \log V - \log V = \log V,
\]
as required.\\

Resuming the proof, we set $T = K/C$ and $V = 2BC$ to obtain the following sufficient condition for \eqref{eq:Klb2}, provided that $\log (2BC) \ge 1$:
\[
K/C \ge 2 \log (2BC) \iff K \ge 2C \log (2BC).
\]

On the other hand, if $\log (2BC) < 1$, then $K / C \ge 1 \implies K/C \ge \log(K/C) + 1 > \log(K/C) + \log (2BC)$, satisfying~\eqref{eq:Klb2}. 

The result follows from these requirements on $K$ together with $K \ge C \log (2A)$ and $K \ge C \log(\frac1\delta)$.

\end{proof}



\subsection{Computation of Inexact Eigenvectors and Eigenvalues}
\label{sec:comp-inex-eigenv}

We include a lemma demonstrating that Assumption~\ref{assump:eigenvector_error}
can be satisfied using matrix-vector products. This result bridges the gap between
iteration complexity and operation complexity, with the latter being defined as
the cumulative count of gradient evaluations and Hessian-vector products. 
The algorithm on which the result is based is a variant of the Lanczos algorithm, tailored to
computing the minimum eigenvalues of a real symmetric matrix \(H\) without
 knowledge of an upper bound on \(\norm{H}\).


\begin{lemma}[{\cite[Lemma~10]{royer2020a-newton-cg-alg}}]
  \label{lemma:Lanczosfixedproba}
  Let $H \in \R^{d \times d}$ be a real symmetric matrix. Let
  $0 < \epsilon_{H} \le 9 \norm{H}$ and $\delta \in (0,1)$ be given. With probability at least
  $1-\delta$ over uniformly random initialization on the unit sphere, there exists an algorithm that on input $H$ (but not $\norm{H}$)
  outputs a unit vector $v$ such that
\begin{equation} \label{eq:approxeigvec}
	v^\top H v \le \lambda_{\min}(H) +  \epsH/ 9
\end{equation}
in at most
\begin{equation} \label{eq:itsLanczos}
	\min \left\{d, 1+ 2\ln \left(25 d / \delta^2\right) \sqrt{{\|H\|}/{\epsilon_{H}}}\right\}
\end{equation}
matrix-vector multiplications by $H$. 
\end{lemma}

In other words, the total number of gradient evaluations and Hessian-vector products required by our algorithm is \(\Order{n \epsH^{-1/2}}\), with \(n\) as specified in~\eqref{eq:hpb}:


\begin{cor}
  \label{thm:high_prob_general-operation}
  Suppose the function \(f\) satisfies Assumption~\ref{assump:basics} and its inexact gradient and Hessian computations satisfy Assumption~\ref{assump:inexactness}. Then, Algorithm~\ref{alg:inexact_randomized} terminates after \(O(n \epsH^{-1/2}) \) gradient evaluations or Hessian-vector products with probability at least \(1-\delta\), where $n$ is defined in \eqref{eq:hpb}. When $\epsg = \epsilon$, $\epsH = \sqrt{\epsilon M}$, and the algorithm takes a short step $\alpha = \epsH$, the operation complexity is $\tOrder{\epsilon^{-9/4}}$, where the dependence on Lipschitz constants is suppressed.
\end{cor}

\section{Comparison with previous results}
\label{sec:roosta-work}

Section~\ref{sec:prior_work} reviewed prior works of finding approximate second-order stationary points with inexact derivatives under varied settings and assumptions.
Here, we revisit the paper \cite{yao2022inexact}, which is especially relevant,
That paper addresses a setting similar to ours, with inexact gradients and Hessians and, in its latter part, no need for function evaluations.
It selects the sense of the negative curvature step \(\inexact p_{k}\) so that \( g_k^T \hat{p}_k \le 0\) and establishes conditions on gradient and Hessian inexactness and target error \(\epsg\) and \(\epsH\)  to ensure sufficient decrease. 
The conditions on {\em Hessian} inexactness are broadly similar between our paper and \cite{yao2022inexact}.
We focus in this section on the {\em gradient} inexactness conditions, showing that our conditions are less stringent than those of \cite{yao2022inexact}.

We denote the bound on \(\|\grad f(x_{k}) - g_{k}\|\) required in the current work by $\Delta_{g,k}$, while the bound required in  \cite{yao2022inexact}  is denoted by $\delta_{g,k}$. 
From Assumption~\ref{assump:inexactness}, we have $\Delta_{g, k} = (1/3)  \max\{\epsilon_{g}, \|g_{k}\|\}$. 
By simplifying \cite[Condition 2.4]{yao2022inexact}, we find that 
\begin{equation}
  \label{eq:gradient_inexactness_roosta}
  \delta_{g,k} \le \frac18 \min \left\{\frac{3\epsilon_{H}^{2}}{65M}, \max\{\epsilon_{g}, \|g_{k}\|\} \right\}.
\end{equation}
It is immediately clear that $\delta_{g,k} < \Delta_{g,k}$, that is, our gradient inexactness requirement in this paper is less stringent. 
In the most useful situations, we may actually have $\delta_{g,k} \ll \Delta_{g,k}$.
Consider first the setting \(\epsH = \sqrt{\epsg M}\) (which is assumed immediately following \cite[Condition 2.4]{yao2022inexact} for the remainder of that paper). By substitution into \eqref{eq:gradient_inexactness_roosta}, we have $\delta_{g,k} \le \frac{3}{520} \epsilon_g$, so we lose the advantage of {\em relative} inexactness allowed by $\Delta_{g,k}$, as well as having a smaller constant factor. 
(Section~\ref{sec:sampling} gives an example in which our relaxed gradient inexactness requirement improves the sample complexity for the finite sum problem~\eqref{eq:fs}.) 
Next, when  the strong coupling between $\epsilon_g$ and $\epsilon_H$ is broken, and we have $\epsilon_H< \sqrt{\epsilon_g M}$, $\delta_{g,k}$ becomes even smaller.
When $\epsilon_g$ and $\epsilon_H$ are chosen so that the two terms in the definition of $\Ceps$ in \eqref{eq:Ceps} are identical, we have  $\epsilon_H = O(\epsilon_g^{2/3})$ and hence $\delta_{g,k} \le O(\epsilon_g^{4/3})$, a more stringent condition.

Next we discuss a low-rank matrix recovery problem in which \(\epsilon_{H}^{3}\) and \( \epsilon_{g}^{2}\) are on the same order.

  Let \(M^{*} \in \R^{d \times d}\) be a real symmetric and positive semidefinite matrix. Suppose the rank of \(M^{*}\) is \(r\), where \(r < d\), and let \(U \in \R^{r \times d}\). Defining the nonconvex objective function \( f(U) = \frac12\fnorm{ UU^{\top} - M^{*} }^{2}\), we seek a global minimum \(\min_{U \in \R^{d \times r}} f(U)\), which then solves the symmetric low-rank matrix factorization problem. 
  Let \(\sigstarl\) and \(\sigstarr\) denote the largest singular value and smallest nonzero singular value of \(M^{*}\), respectively, and define  \(\kappa = \sigstarl / \sigstarr\) to be the condition number. 
  (We abuse the notation in the following discussion and let \(f(\vect(U)) = f(U)\) so that the gradient \(\grad f(U) \in \R^{dr}\) remains a vector and the Hessian \(\grad^{2}f(U) \in \R^{dr \times dr}\) is a matrix.)

   The following strict saddle property gives a context in which  approximate second-order
  stationary points are in the neighborhood of local minima.
  \begin{defn}
  Function \(f(\cdot)\) is \((\epsg, \epsilon_{H}, \zeta)\)-\emph{strict saddle} if, for any $U$, at least one of following holds:
    \begin{itemize}
      \item $\norm{\grad f(U)} \ge \epsilon_{g}$.
      \item $\lambda_{\min}(\grad^{2} f(U)) \le -\epsilon_{H}$.
      \item $U$ is $\zeta$-close to $\cXstar$ --- the set of local minima. 
    \end{itemize}
  \end{defn}

   This property provides a guarantee that a point is
  \(\zeta\)-close to the region of local optima if it is an
  \((\epsilon_{g}, \epsilon_{H})\)-approximate second-order stationary point.

  The matrix recovery literature, for instance~\cite{jin2017howto}, provides
  the constants for the strict saddle property of \(f(\cdot)\) and its
  gradient Lipschitz and Hessian Lipschitz constants.

  \begin{fact}
    \label{fact:matrix_factorization}
    All local minima of the function \( f(\cdot)\)  are global minima. Moreover, $f(\cdot)$ satisfies
    $(\frac{1}{24}(\sigstarr)^{3/2}, \frac{1}{3}\sigstarr, \frac{1}{3}(\sigstarr)^{1/2})$-strict
    saddle property.  For any \(\Gamma > \sigstarl\), inside the region \(\{U: \norm{U}^{2} < \Gamma\}\), \( f(\cdot)\) has gradient Lipschitz constant \(L = 16\Gamma\) and Hessian Lipschitz constant \(M = 24\Gamma^{1/2}\).  
  \end{fact}

  An approximate second-order stationary point is useful for this problem because the function \( f\)  satisfies a regularity condition similar to strong convexity inside \(\frac{1}{3}\sigstarr^{1/2}\)-neighborhood of \(\cXstar\), so gradient descent initialized at an $(\epsilon_g,\epsilon_H)$-approximate second-order point, where
  \begin{equation}
  \label{eq:example_epsilon}
  \epsilon_{g} = \frac{1}{24}\sigstarr^{3/2}, \quad \epsilon_{H} =  \frac{1}{3}\sigstarr,
\end{equation}
will achieve local linear convergence~\cite[Section 3.2]{jin2017howto} to a local minimum.
  That is, once we find such an approximate second-order point, the additional cost of finding a solution to high accuracy is relatively low. 
  Note also that \eqref{eq:example_epsilon} satisfies the assumption \(M \epsg < L \epsH\) required in Corollary~\ref{cor:short-step}, so the iteration complexity for finding a second-order stationary point is  $\tilde O(\epsg^{-2})$, which in this case is a multiple of $(\sigstarr)^{-3}$.

Let us compare the gradient inexactness conditions between our work in~\cite{yao2022inexact} that are required to find an \((\epsg, \epsH)\)-approximate second-order point for $\epsg$ and $\epsH$ defined in \eqref{eq:example_epsilon}. 
For Algorithm~\ref{alg:inexact_randomized}, the inexactness bound $\Delta_{g,k}$ satisfies
\[
\Delta_{g, k} = \frac13 \max\{\epsg, \norm{g_{k}}\} = \frac13 \max \left\{\frac{1}{24}\sigstarr^{3/2}, \norm{g_{k}} \right\}
\]

By contrast, for the inexactness bound in~\cite{yao2022inexact}, using $M = 24 \Gamma^{1/2} > 24 (\sigma_1^*)^{1/2}$, we have
\begin{align*}
  \delta_{g,k} & \le \frac18 \min \left\{ \frac{3\epsilon_{H}^{2}}{65M}, \max\{\epsilon_{g}, \|g_{k}\|\} \right\} \\
& \le \frac18 \min \left\{ \frac{\sigstarr^{2}}{(3 \times 65 \times 24) (\sigma_1^*)^{1/2}}, \max\left\{\frac{1}{24} \sigstarr^{3/2}, \|g_{k}\|\right\} \right\} \\
& \le \frac18 \min \left\{ \frac{\sigstarr^{3/2}}{4680\sqrt{\kappa}}, \max\left\{\frac{1}{24} \sigstarr^{3/2}, \|g_{k}\|\right\} \right\} = \frac{\sigstarr^{3/2}}{37440\sqrt\kappa}.
\end{align*}
We see that $\delta_{g,k}$ is smaller than $\Delta_{g,k}$ by a factor of at least {$520\sqrt{\kappa}$} where \(\kappa = \sigstarl/\sigstarr\), and even smaller in the usual regime of  $\norm{g_k}> \sigstarr^{3/2}/24$.
Thus, in a {stochastic setting (see Section~\ref{sec:sampling}) or an outlier-robust setting (discussed in~\cite{li2023robust-second-order})}, many more samples are required to evaluate $g_k$ to the desired accuracy.





\section{Sampling}\label{sec:sampling}
We consider the finite sum problem \eqref{eq:fs} and investigate the number of samples that are necessary to construct inexact gradient and Hessian oracles that satisfy Assumption~\ref{assump:inexactness}. 
We focus on sample complexity for the gradient inexactness condition of this paper which, as noted above, is less stringent than for previous methods. 
(As noted at the start of Section~\ref{sec:roosta-work}, the Hessian inexactness condition resembles those of previous works, so such results as \cite[Lemma~16]{xu2020newton} and \cite[Lemma~4]{tripuraneni2017stochastic} capture the  sample complexity for Hessian inexactness in this paper too.)

In this section, we give sample complexity results for estimating the gradient and Hessian to the required accuracy at each point.
We then give a result for total sample-operation complexity, which is a high-probability bound for the total number of stochastic gradient or Hessian-vector product evaluations on a single sample required by the algorithm to find the $(\epsg,\epsH)$-approximate second-order point.
We conclude with a brief comparison with other algorithms.

We have the following sample complexity result for gradient subsampling.
The proof closely tracks that of \cite[Lemma~2.13]{yao2022inexact}, replacing $\epsg$  by $\max(\epsg, \| \nabla f(x) \|)$ throughout.
\begin{lemma}[Gradient Subsampling]
  \label{lemma:gradient_subsampling}
  For a given \(x \in \R^{d}\), suppose there is an upper bound \(G(x)\) such that \(\| \grad f_{i}(x)\| \le G(x) < \infty \) for all \(i \in [N]\). For any given \(\xi \in (0,1)\), let \(S_{g}\) be a set of indices  that is sampled uniformly with replacement from $\{1,2,\dotsc,N\}$, with
  \begin{equation}
    \label{eq:gradient-subsampling}
    |S_{g}| \ge 256 \left( \frac{G(x)}{\max\{\epsilon_{g}, \norm{\grad f(x) }\}} \right )^{2} \log \frac1{\xi}.
  \end{equation}
  Then for  \(g(x):=\frac{1}{|S_{g}|} \sum_{i \in S_{g}} \grad f_{i}(x)  \), we have
  \[\Prob\left(  \|\grad f(x) - g(x)\| \le \frac13 \max\{\epsilon_{g}, \norm{g(x) }\}  \right ) \ge 1 - \xi.\]
\end{lemma}
\begin{proof}
  If we can prove that
  \begin{equation}
    \label{eq:deterministic_rhs}
     \|\grad f(x) - g(x)\| \le \frac14 \max\{\epsilon_{g}, \norm{\grad f(x) }\} =: \Delta(\epsg, x),
  \end{equation}
  then it follows from the triangle inequality that  \( \|\grad f(x) - g(x)\|  \le \frac13 \max\{\epsilon_{g}, \norm{g(x) }\}\).
  We can therefore prove the result by showing that \eqref{eq:deterministic_rhs} holds with probability at least \(1-\xi\). The required sample size is shows by \cite[Lemma 2.12]{yao2022inexact} to be 
  \[ \abs{S_{g}} \ge 16 \frac{G(x)^{2}}{\Delta(\epsg, x)^{2}} \log\frac1\xi,\]
  which yields the result when we substitute the definition of $\Delta(\epsg, x)$.
\end{proof}

\begin{remark}
  In the context of Remark~\ref{remark:random_inexactness}, we may need to choose $\xi$ such that  \(\xi = \delta/n = O(\delta^{2}\epsilon_{g}^{2})\) to achieve a failure probability of at most $\delta$. 
  This choice causes the modest factor of $(\log \delta + \log \epsg)$ to appear in the lower bound on $|S_g|$. We formally establish this observation in Theorem~\ref{thm:total-operation-complexity} with a tighter bound.
\end{remark}

\begin{remark}[Alternative assumptions]
  If we assume in addition to the assumption of Lemma~\ref{lemma:gradient_subsampling} (namely, \(\| \grad f_{i}(x)\| \le G(x) < \infty \) for all \(i \in [N]\)) that there is an upper bound \(V_{g}(x)\) on the variance of the sample gradients,  that is, 
  \[ \E_{i} \norm{\grad f_{i}(x) - \grad f(x)}^{2} \le V_{g}(x),\]
  we can use the matrix Berstein inequality to refine the lower bound on $|S_g|$ to 
  \[
  \tilde{O} \left( \max\left\{ \frac{V_{g}(x)}{\max\{\epsilon_{g}^{2}, \norm{\grad f(x) }^{2}\}}, \frac{G(x)}{\max\{\epsilon_{g}, \norm{\grad f(x) }\}} \right\} \right),
  \]
  by following the approach of \cite[Lemma~4]{tripuraneni2017stochastic}). 
  If we only assume the second-moment upper bound \(V_{g}(x)\),  without  a bound \(G(x)\) on sample gradient norms,  then we need to employ
  Chebyshev's inequality to obtain the sample complexity
  \( O \left( \frac{V_{g}(x)}{\xi^{2}\max\{\epsilon_{g}, \norm{\grad f(x) }\}^2} \right) \),
  whose polynomial dependence on \(1/\xi\) makes it much less appealing.
\end{remark}


\begin{remark}
  With the same analysis method, \cite[Lemma 2.13]{yao2022inexact} derived the
  following gradient subsampling
  bound \[|S_{g}| \ge 480 000 \left( \frac{G(x)}{\epsilon_{g}} \right)^{2} \log \frac{1}{\xi}. \]
  Lemma~\ref{lemma:gradient_subsampling} improves on this bound via a better constant factor and relative gradient inexactness.
\end{remark}



\begin{remark}[Adaptive sample size]\label{remark:adaptive-sample}
  The lower bound on $|S_g|$ cannot quite be evaluated in practice because we do not know $\| \nabla f(x) \|$  (only the estimate $\| g(x) \|$). A similar issue was tackled by \cite[{Section 5.1}]{cartis2018global-convergence}, who extended the strategy in~\cite{byrd2012sample} to develop an adaptive
  algorithm for the finite-sum problem~\eqref{eq:fs} that exploits relative
  gradient inexactness to reduce the number of samples required when less
  accurate gradients are permissible (i.e., when gradients are large), so the total sample complexity across all iterations is guaranteed to be smaller.
\end{remark}

We can also utilize subsampling to estimate Hessians.
As mentioned in the beginning of this section, our requirements on Hessian accuracy are similar to those of previous works, and so is the following sample complexity result, which is stated for completeness.
\begin{lemma}[Hessian Subsampling~{\cite[Lemma 16]{xu2020newton}}]
  \label{lemma:hessian_subsampling}
  For a given \(x \in \R^{d}\), suppose there is an upper bound \(K(x)\) such that \(\| \grad^{2} f_{i}(x)\| \le K(x) < \infty \) for all \(i \in [N]\). For any given \(\xi \in (0,1)\), if
  \begin{equation}
    \label{eq:hessian-subsampling}
    |S_{H}| \ge 484  \left( \frac{K(x)}{\epsH} \right )^{2} \log\frac{2d}{\xi},
  \end{equation}
  where \(S_{H}\) is the set of indices, sampled with- or without-replacement from $[N]$, then for  \(H(x):=\frac{1}{|S_{H}|} \sum_{i \in S_{H}} \grad^{2} f_{i}(x)  \), we have
  \[\Prob\left(  \|\grad^{2} f(x) - H(x)\| \le \frac29 \epsH \right ) \ge 1 - \xi.\]
\end{lemma}

Our total sample-operation complexity result is as follows.
  \begin{theorem}[Total Sample-Operation Complexity]
    \label{thm:total-operation-complexity}
  Consider the finite-sum problem~\eqref{eq:fs} and apply Algorithm~\ref{alg:inexact_randomized} with the settings of Corollary~\ref{cor:short-step}. 
  Choose $\delta \in (0,.5]$.
  Let the inexact gradient and Hessian oracles be obtained by independent with-replacement batch subsampling, that is, 
  \(g(x):=\frac{1}{|S_{g}|} \sum_{i \in S_{g}} \grad f_{i}(x)\) and
  \(H(x):=\frac{1}{|S_{H}|} \sum_{i \in S_{H}} \grad^{2} f_{i}(x)\), where
  \(\abs{S_{g}}\) and \(\abs{S_{H}}\) are given by
  \eqref{eq:gradient-subsampling} and~\eqref{eq:hessian-subsampling}, respectively, with \(\xi = \delta / n\) and \(n\) as defined in Theorem~\ref{thm:high_prob_general} with the parameter settings of Corollary~\ref{cor:short-step}. Then with probability at least \(1 - 2\delta\), Algorithm~\ref{alg:inexact_randomized} will output an
  \((\frac43 \epsg, \frac43 \epsH)\)-approximate second-order stationary points
  within
  \[ \bigtO{\max\{\epsg^{-4}, \epsg^{-2}\epsH^{-3},\epsH^{-11/2}, \epsg^{2}\epsH^{-13/2} \}}\] total sample-operations (defined as the total number of stochastic gradient or Hessian-vector product evaluations for single functions \(f_{i}\)'s).
\end{theorem}

\begin{proof}
By Lemmas~\ref{lemma:gradient_subsampling} and~\ref{lemma:hessian_subsampling}, it follows from a union bound over all \(n\) iterations that with probability at least \(1 - \xi n = 1 - \delta\), Assumption~\ref{assump:inexactness} is satisfied. Conditioning on this event, Corollary~\ref{cor:short-step} implies that Algorithm~\ref{alg:inexact_randomized} terminates and output an \((\frac43 \epsg, \frac43 \epsH)\)-approximate second-order stationary points within at most \(n\) iterations with probability at least \(1-\delta\). Therefore, the total failure probability is no larger than \(2\delta\).

  We calculate the number of samples and total sample-operations carefully. 
  Recall that Theorem~\ref{thm:high_prob_general} and Corollary~\ref{cor:short-step} imply that \(n  = \tilde{O} (\max\{\epsg^{-2}, \epsH^{-3}\} + \frac{\max\{\epsH,\epsg\}\epsg}{\epsH^4}\log\frac1\delta ) \).
  For gradient steps, the number of samples for each iteration is
  \begin{align*}
    \abs{S_{g}}
    &= \bigO{\epsg^{-2}\log(1/\xi)} = \bigO{\epsg^{-2}\log(n/\delta)} = \bigtO{\epsg^{-2}}
  \end{align*}
as $\tilde O$ suppresses factors of $\poly\log (\epsg)$, $\poly\log(\epsH)$, and $\log \delta$, so the number of total sample-operations contributed by the gradient steps is at most \( \tilde{O} (\max\{\epsg^{-4}, \epsg^{-2}\epsH^{-3}, \epsH^{-4}\} ) \).  Similarly, for negative curvature steps, the number of samples is \(\abs{S_{H}} = \bigtO{\epsH^{-2}}\). By Lemma~\ref{lemma:Lanczosfixedproba}, each iteration of negative curvature steps needs \(\tilde{O} (\epsH^{-1/2})\) Hessian-vector products per sample, so the number of total sample-operations from negative curvature steps is at most \(\tilde{O} (\max\{\epsH^{-11/2}, \epsg^{-2}\epsH^{-5/2}, \epsg^{2}\epsH^{-13/2}\})\).

\end{proof}

\begin{remark}
  We match the \(\bigtO{\epsg^{-4}}\) total sample-operation complexity of perturbed (stochastic) gradient descent methods~\cite{ge2015escaping,jin2021nonconvex} when \(\epsH = \sqrt{\epsg}\).
  As mentioned in Section~\ref{sec:prior_work}, our algorithm considers general inexact settings and does not take advantage of the stochastic or finite-sum structure of the problem, which (partly) explains why our total sample-operation complexity is worse than the optimal rate  $O(\max(\epsg^{-3},\epsH^{-5}))$ established in~\cite{arjevani2020second-order-in} under stochastic settings. However, we note that our gradient sample complexity is an overestimate because of our relative tolerances on gradient accuracy, and we leave it to future work on how the total sample-operation complexity can be reduced by adaptively adjusting the sample size as discussed in Remark~\ref{remark:adaptive-sample}.

  Among algorithms that tackle general inexact settings, we also obtain worse iteration complexity than the \(O(\epsg^{-3/2})\) rate of~\cite{yao2022inexact,tripuraneni2017stochastic}.
  However, apart from the advantage that we tolerate more inexactness, we also highlight that our method is much simpler than capped conjugate gradient method and stochastic cubic regularization.
\end{remark}



\section{Conclusion}\label{sec:conclusion}
We have presented a simple randomized algorithm that finds an \((\epsg, \epsH)\)-approximate second-order stationary point of a smooth nonconvex function with inexact oracle access to gradient and Hessian.
Compared to previous studies, our method's main advantage is its enhanced tolerance for inexactness. Specifically, the gradient inexactness measure can be relative, and no coupling between the parameters \(\epsg\) and \(\epsH\) that define accuracy in first- and second-order conditions is required. Both features make our inexactness criteria easier to satisfy, as evidenced in the low-rank matrix factorization example of Section~\ref{sec:roosta-work} and the subsampling complexity result for finite-sum problems in Section~\ref{sec:sampling}. Moreover, our approach eliminates the need for approximate evaluations of the function value, which typically come with stringent requirements in terms of precision and therefore high sample complexity in finite-sum settings.


\printbibliography
\end{document}